\documentclass[a4paper,reqno]{amsart}
\usepackage{amssymb, amsmath, amscd}

\def\XIx\langle#1\rangle{h(#1)}

\newtheorem{theorem}{Theorem}[section]
\newtheorem{definition}[theorem]{Definition}
\newtheorem{lemma}[theorem]{Lemma}
\newtheorem{conjecture}[theorem]{Conjecture}
\newtheorem{proposition}[theorem]{Proposition}
\newtheorem{remark}[theorem]{Remark}
\newtheorem{corollary}[theorem]{Corollary}

\def\qedbox{\hbox{$\rlap{$\sqcap$}\sqcup$}}
\def\Tr{\operatorname{Tr}}

\makeatletter
 \@addtoreset{equation}{section}
\makeatother
\begin{document}
\title[Growth and decay of heat trace asymptotic coefficients]{Growth of heat trace coefficients\\ for locally
symmetric spaces}
\author{P. Gilkey}
\address{Mathematics Department, University of Oregon, Eugene OR 97403 USA}
\email{gilkey@uoregon.edu}
\author{R.\@ J. Miatello}
\address{FaMAF-CIEM, Universidad Nacional de C\'ordoba,
Argentina}
\email{miatello@famaf.unc.edu.ar}

\begin{abstract}{We study the asymptotic  behavior of the heat
trace coefficients $a_n$ as $n\rightarrow\infty$ for the scalar
Laplacian in the context of locally symmetric spaces. We show that if the
Plancherel measure of a noncompact type symmetric space is polynomial, 
then these coefficients are $O(\frac1{n!})$. On the other hand, for even dimensional
locally rank 1-symmetric spaces, one has
$|a_n|\approx C^n\cdot n!$; we conjecture this is the case in general if the associated Plancherel measure is
not polynomial. These examples show that growth estimates conjectured by Berry and
Howls \cite{BH94} are sharp. We also construct examples of locally symmetric
spaces which are not irreducible, which are not flat, and so that
only a finite number of the $a_n$ are non-zero.}\end{abstract}
\maketitle
 
\section{Introduction and Statement of Results}

Let $\Delta_{\mathcal{M}}$ be the Laplace-Beltrami operator of a
compact connected Riemannian manifold $\mathcal{M}:=(M,g)$ without
boundary of dimension $m\ge2$. The fundamental solution of the
heat equation, $e^{-t\Delta_{\mathcal{M}}}$, is of trace class in
$L^2$ and as $t\downarrow0$ there is a complete asymptotic
expansion with locally computable coefficients of the form:
$$\operatorname{Tr}_{L^2}(e^{-t\Delta_{\mathcal{M}}})=
(4\pi t)^{-m/2}\sum_{n=0}^N a_n(\mathcal{M})t^n+O(t^{-m/2+N})\quad\text{for any}\quad N\,.$$
These coefficients are well known; see, for example, the
discussion in \cite{AmBeOc89,G95} and the references therein.
In particular $a_0(\mathcal{M})=\operatorname{Vol}(\mathcal{M})$. 

\subsection{Growth estimates for heat trace asymptotics} Berry
and Howls \cite{BH94} examined the heat trace coefficients for a
real analytic domain $\Omega$ in $\mathbb{R}^2$ where $D$ was the
Dirichlet Laplacian and conjectured there were growth estimates
of the nature $a_n(\Omega)\approx C(\Omega)^n\cdot n!$ in that
context. Inspired by the work of Howls and Berry, Trav\v{e}nec
and \v{S}amaj \cite{TS11} got similar factorial growth for the
heat content asymptotics of real analytic domains in Euclidean
space in certain settings.  

We shall assume for the remainder of this paper that $\mathcal{M}$ is a
compact connected Riemannian manifold without boundary of dimension
$m\ge2$. The following result was established by van den Berg et.\! al
\cite{BGK11} using the Seeley calculus
\cite{Se68b,Se69a}:

\begin{theorem}\label{thm-1.1}
If $\mathcal{M}$ is real analytic, then there exists a
constant $C_1=C_1(\mathcal{M})$ so that 
$|a_n(\Delta_{\mathcal{M}})|\le 
C_1^{n}\cdot n!\cdot\operatorname{Vol}(\mathcal{M})$ for any $n$.
\end{theorem}

The existence of manifolds where a similar lower bound held was
left open in \cite{BGK11} and formed the initial focus of
investigation for this paper; it will follow from Theorem
\ref{thm-1.4} below that Theorem~\ref{thm-1.1} is sharp in this
regard. We note there are no universal growth estimates
available in the smooth context.  If $h\in C^\infty(M)$, let
$\mathcal{M}_h:=(M,e^{2h}g)$ be the conformally adjusted Riemannian
manifold. One has \cite{BGK11}:

\goodbreak\begin{theorem} Let constants $C_n>0$ for
$n\ge3$ be given.  If $\mathcal{M}$ is only assumed to be smooth, then
there exists $h\in C^\infty(M)$ so that
$|a_n(\mathcal{M}_h)|\ge C_n$ for any $n\ge3$.
\end{theorem}

\subsection{Symmetric spaces} We shall assume henceforth that
$\mathcal{M}=(M,g)$, is locally a symmetric space, i.\!\,e.\! that the
covariant derivative $\nabla R$ of the curvature tensor vanishes.
Consequently,
$\mathcal{M}$ is real analytic and
$\mathcal{M}$ is locally isometric to a quotient of the form $G/K$ 
for some suitable subgroup $K$ of a Lie group $G$. Thus
$\mathcal{M}$ is locally homogeneous and we shall say that
$\mathcal{M}$ is {\it modeled} on $G/K$. There is a discrete, cocompact
subgroup $\Gamma$ of $G$ so that
$\mathcal{M}=\Gamma\backslash G/K$. Let 
$$\mathcal{A}_n(\mathcal{M}):
=\frac{a_n(\Delta_{\mathcal{M}})}{\operatorname{Vol}(\mathcal{M})}\,.$$
With this normalization, $\mathcal{A}_0(\mathcal{M})=1$ and
$\mathcal{A}_n(\mathcal{M})$ only depends on the local isometry
type of $\mathcal{M}$ and is determined by the model, $G/K$. We
consider the formal power series
\begin{equation}\label{eqn-1.a} 
\mathcal{H}_{\mathcal{M}}(t):
=\sum_{n=0}^\infty \mathcal{A}_n(\mathcal{M})t^n\,.
\end{equation} 

\subsection{Rank 1-symmetric spaces}
The complete simply connected rank $1$-sym\-metric spaces are
classified; see Proposition~\ref{prop-3.1} for details. The only
odd dimensional examples are modeled on the spheres and hyperbolic spaces.
Before stating the next result, we define the following equivalence relation:

\begin{definition}\label{defn-1.3} Let $\Xi:=\{\Xi_n\}_{n\ge0}$ and
$\tilde\Xi:=\{\tilde\Xi_n\}_{n\ge0}$ be two infinite sequences of real
numbers which are positive for $n$ large. We say that
$\Xi\approx\tilde\Xi$ if given any
$0<\varepsilon<1$, there exists $N(\varepsilon)\in\mathbb{N}$ so that
$$\Xi_n(1-\varepsilon)^n<\tilde\Xi_n<\Xi_n(1+\varepsilon)^n\quad\text{for}\quad
n\ge N(\varepsilon)\,.$$
\end{definition}

 We note that Definition~\ref{defn-1.3} gives a fairly crude growth estimate since,
for example, $\{14n^2\cdot n!\}\approx\{n!\}$; thus multiplicative constants and finite
powers of
$n$ are suppressed; this will simplify the exposition. The following result will be
proved in Section~\ref{sect-3}; it shows that Theorem
\ref{thm-1.1} is sharp:

\begin{theorem}\label{thm-1.4}
Let $\mathcal{M}$ be modeled on an even dimensional
irreducible rank 1 symmetric space. There exists
$C=C(\mathcal{M})>0$ so that 
$$\left\{|\mathcal{A}_n(\mathcal{M})|\vphantom{\vrule height 10pt}\right\}_{n\ge0}\approx
\left\{C^n\cdot n!\vphantom{\vrule height 10pt}\right\}_{n\ge0}\,.$$
\end{theorem}

\subsection{Properties of the heat trace asymptotics}
We will establish the following result in Section \ref{sect-2}.
Assertion (1) will let us assume $\mathcal{M}$ is modeled on
an irreducible symmetric space, Assertion (2) will permit us to
rescale the metric, and Assertion (3) is the duality result of Cahn and Wolf
\cite{CW75} that will let us pass between models of
non-compact type and models of compact type:
\begin{lemma}\label{lem-1.5}
Let $\mathcal{M}=(M,g)$ be modeled on a symmetric space.
\begin{enumerate}
\item If $\mathcal{M}$ is modeled on a product
$\mathcal{M}_1\times...\times\mathcal{M}_k$ of symmetric
spaces, then
$\mathcal{H}_{\mathcal{M}}(t)
=\mathcal{H}_{\mathcal{M}_1}(t)\cdot\cdot\cdot
\mathcal{H}_{\mathcal{M}_k}(t)$.
\smallbreak\item If $0\ne c\in\mathbb{R}$, then
$\mathcal{H}_{(M,c^{-2}g)}(t)=\mathcal{H}_{(M,g)}(c^2t)$.
\smallbreak\item If $\mathcal{M}$ is modeled on a symmetric
space of non-compact type and if $\tilde{\mathcal{M}}$ modeled on
the dual symmetric space of compact type, then
$\mathcal{H}_{\tilde{\mathcal{M}}}(t)
=\mathcal{H}_{\mathcal{M}}(-t)$.
\end{enumerate}\end{lemma}

\subsection{Locally symmetric spaces where the heat trace asymptotics decay}

The Plancherel measure plays a crucial role in the analysis.
Section~\ref{sect-4} is devoted to the proof of the following
result:

\begin{theorem}\label{thm-1.6}
Let $ G/K$ be an irreducible symmetric space type of non-compact
type and let $ \tilde G/\tilde K$ be the associated dual symmetric
space of compact type. Assume the Plancherel measure of $G$ is
polynomial. Let $\mathcal{H}_{\mathcal{M}}(t)$ be as in  Equation~\eqref{eqn-1.a}. Then there
exists
$\kappa_{G/K}<0$ and a polynomial $\mathcal{P}_{G/K}(t)$ so that:
\begin{enumerate}
\item If $\mathcal{M}$ is modeled on $G/K$, then
$\mathcal{H}_{\mathcal{M}}(t)
=e^{k_{G/K}t}\cdot\mathcal{P}_{G/K}(t)$.
\item If $\tilde{\mathcal{M}}$ is modeled on
$\tilde G/\tilde K$, then
$\mathcal{H}_{\tilde{\mathcal{M}}}(t)=e^{-k_{G/K}t}\cdot
\mathcal{P}_{G/K}(-t)$.
\end{enumerate}\end{theorem} 

Proposition~\ref{prop-4.2} will list the models which can occur in
Theorem~\ref{thm-1.6} and will give the associated Plancherel measures; we
postpone the discussion until Section~\ref{sect-4} to establish the requisite
notation. The odd dimensional spheres and projective spaces are of the form
given in Theorem~\ref{thm-1.6} so this completes our discussion of manifolds
which are modeled on the simply connected irreducible rank 1-symmetric spaces.

We can construct examples of non-flat manifolds so that
$a_n(\mathcal{N})=0$ for $n$ large. We apply Lemma~\ref{lem-1.5} to see:

\begin{corollary}
Let $\mathcal{M}$ and $\tilde{\mathcal{M}}$ be as in
Theorem~\ref{thm-1.6}. Set
$\mathcal{N}:=\mathcal{M}\times\tilde{\mathcal{M}}$; 
this is, of course, not irreducible. Then:
$$\mathcal{H}_{\mathcal{N}}(t)=
e^{k_{\mathcal{M}}t}e^{-k_{\mathcal{M}}t}
\mathcal{P}_{\mathcal{M}}(t)\mathcal{P}_{\mathcal{M}}(-t)
=\mathcal{P}_{\mathcal{M}}(t)\mathcal{P}_{\mathcal{M}}(-t)\,.$$
\end{corollary}

 In particular, we could take
$\mathcal{M}=\Gamma\backslash H^3$ to have constant curvature $-1$ and
$\tilde{\mathcal{M}}=S^3$ to have constant
curvature $+1$. We shall see presently that
$\mathcal{H}_{S^3}(t)=e^{t/4}$ and
$\mathcal{H}_{\mathcal{H}^3}(t)=e^{-t/4}$. Consequently
$\mathcal{H}_{\mathcal{N}}=1$, so $a_n(\mathcal{N})=0$ for
$n\ge1$.  

Theorem~\ref{thm-1.4} and Theorem~\ref{thm-1.6} lead us to make
the following:

\begin{conjecture}
Let $\mathcal{M}$ be modeled on an irreducible simply connected
symmetric space $G/K$ of non-compact type. If the Plancherel
measure of $G$ is not polynomial, 
then there exists
$C=C(\mathcal{M})$ so that
$|\mathcal{A}_n(\mathcal{M})|\ge C^n\cdot n!$
 for $n$ sufficiently large.
\end{conjecture}

\section{Proof of Lemma~\ref{lem-1.5}}\label{sect-2}
 We have that
$\Delta_{\mathcal{M}_1\times\mathcal{M}_2}=\Delta_{\mathcal{M}_1}\oplus\Delta_{\mathcal{M}_2}$. Consequently:
\begin{eqnarray*}
&&\Tr_{L^2}\left\{e^{-t(\Delta_{\mathcal{M}_1
\times\mathcal{M}_2})}\right\}
=\Tr_{L^2}\left\{e^{-t(\Delta_{\mathcal{M}_1}
\oplus\Delta_{\mathcal{M}_2})}\right\}\\
&=&\Tr_{L^2}\left\{e^{-t(\Delta_{\mathcal{M}_1})}\right\}\cdot
\Tr_{L^2}\left\{e^{-t\Delta_{\mathcal{M}_2})}\right\}\,.
\end{eqnarray*}
Since $(4\pi t)^{-(m_1+m_2)/2}=(4\pi t)^{-m_1/2}(4\pi t)^{-m_2/2}$ and since $dx=dx_1\cdot dx_2$, we may prove 
Lemma~\ref{lem-1.5}~(1) by
equating terms in the asymptotic expansions. Because
$\Delta_{(M,c^{-2}g)}=c^2\Delta_{(M,g)}$, we have that:
$$\displaystyle\Tr_{L^2}
   \left\{e^{-t\Delta_{(M,c^{-2}g)}}\right\}
=\displaystyle\Tr_{L^2}\left\{e^{-c^2t\Delta_{(M,g)}}\right\}
\,.$$
Since $(4\pi t)^{-m/2}dx_g=(4\pi c^2t)^{-m/2}dx_{c^{-2}g}$,
Lemma~\ref{lem-1.5}~(2) follows by equating terms in the
asymptotic expansions.

The heat trace coefficients are given by integrating local
invariants $a_n(\cdot,\Delta)$. These are invariant
expressions which are homogeneous of order
$2n$ in the derivatives of the metric or, equivalently, in the
curvature tensor $R$, in the covariant derivative of the
curvature tensor $\nabla R$, and so forth. If $\mathcal{M}$ is
modeled on a symmetric space, then $\nabla^kR=0$ for any $k>0$
and thus $a_n(\cdot,\Delta)=a_n(R_{ijkl})$ is a polynomial which
is homogeneous of degree $n$. Since
$R_{\mathcal{M}}=-R_{\tilde{\mathcal{M}}}$ where
$\tilde{\mathcal{M}}$ is modeled on the dual symmetric space,
Lemma~\ref{lem-1.5}~(3) follows.\hfill\qedbox

\section{The proof of Theorem~\ref{thm-1.4}}\label{sect-3}

 The difference between $\mathcal{A}_n$ and $a_n$ lies in the multiplicative constant
$\operatorname{Vol}(\mathcal{M})$. Since the equivalence relation of Definition~\ref{defn-1.3} is
not sensitive to multiplicative constants, we will work with $a_n$ rather than $\mathcal{A}_n$.
\subsection{The classification of rank 1-symmetric spaces}

\begin{proposition}\label{prop-3.1}
Let $\mathcal{M}$ be an irreducible simply connected rank 1-symmetric
space. Then $\mathcal{M}$ is, up to homothety, one of the following
examples:\begin{enumerate}
\item Compact type:
\begin{enumerate}
\item The sphere $S^{\bar m}$ of radius $1$ in $\mathbb{R}^{\bar m+1}$.
\item The complex projective space $\mathbb{CP}^{\bar m}$ with the
Fubini-Study metric.
\item The quaternionic projective space $\mathbb{HP}^{\bar m}$
with the Fubini-Study metric.
\item The Cayley projective plane $\mathbb{OP}^2$ with the canonical metric.
\end{enumerate}
\item Non-compact type:
\begin{enumerate}
\item  The real hyperbolic space ${\bf H}^{\bar m}$ of
constant sectional curvature
$-1$. This is the non-compact type dual of the $m$-sphere.
\item The complex hyperbolic space $\mathbb{C}{\bf H}^{\bar m}$, non-compact type dual of the complex
projective $\bar m$-space.
\item The quaternionic hyperbolic space $\mathbb{H}{\bf H}^{\bar m}$, non-compact type dual of the  quaternionic
projective $\bar m$-space.
\item The Cayley hyperbolic plane  $\mathbb{O}{\bf H}^2$, non-compact type dual of the Cayley projective plane.
\end{enumerate}
\end{enumerate}\end{proposition}

We may apply Lemma~\ref{lem-1.5}~(2) to see that the estimates of
Theorem~\ref{thm-1.4} are unchanged by homothety and therefore
to assume that the curvature of $\mathcal{M}$ is standard. We use
Lemma~\ref{lem-1.5}~(3) to assume $\mathcal{M}$ is of compact
type.  We will then proceed on a case by case basis to prove
Theorem~\ref{thm-1.4} using the classification of
Proposition~\ref{prop-3.1}. Section~\ref{sect-3.2} deals with the
even dimensional spheres, Section~\ref{sect-3.3} deals with the
complex projective spaces, Section~\ref{sect-3.4} deals with the
quaternionic projective spaces, and Section~\ref{sect-3.5} deals
with the Cayley plane; the odd dimensional spheres are treated in
Theorem~\ref{thm-1.6}. We shall use results of
\cite{CW76} in Section~\ref{sect-3}; results of \cite{M80}  will
play a prominent role in the analysis of Section~\ref{sect-4}.

\subsection{Even dimensional spheres}\label{sect-3.2}
Theorem~\ref{thm-1.4} for the round sphere will follow in
this case from the following result:
\begin{lemma}
Let $\mathcal{M}$ be the sphere of radius $1$ in
$\mathbb{R}^{2\bar m+1}$. Then
$$\left\{(-1)^{\bar m-1}a_n(\mathcal{M})\vphantom{\vrule height
10pt}\right\}_{n\ge0}\approx
\left\{\frac{(\bar m-\frac12)^n}{\pi^{2n}\cdot 4^n}\cdot n!\right\}_{n\ge0}\,.$$
\end{lemma}

\begin{proof}
Recall that the Bernoulli numbers were expressed by
Euler in terms of the Riemann zeta function in the form
\cite{W1}:
$$
B_{2n}=(-1)^{n+1}2(2\pi)^{-2n}(2n)!
\left\{1+2^{-2n}+3^{-2n}+4^{-2n}+...\right\}
$$
Following \cite{CW76} (see page 12) one defines
$$
c_n:=(-1)^n(n+1)^{-1}B_{2n+2}(1-2^{-2n-1})\,.
$$
Stirling's formula \cite{W2} yields
$$\left\{n!\vphantom{\vrule height 10pt}\right\}_{n\ge1}\approx\left\{\sqrt{2\pi
n}\left(\frac ne\right)^n\right\}_{n\ge1}\approx\left\{\left(\frac
ne\right)^n\right\}_{n\ge1}\,.$$ Combining these results permits us to compute:
\begin{eqnarray*}
\left\{c_n\approx(2\pi)^{-2n}(2n)!\vphantom{\vrule height
10pt}\right\}_{n\ge0}&\approx&
\left\{(2\pi)^{-2n}2^{2n}\left(\frac ne\right)^n\cdot\left(\frac
ne\right)^n\vphantom{\vrule height 10pt}\right\}_{n\ge0}
\\&\approx&\left\{\pi^{-2n}\cdot n!\cdot n!\vphantom{\vrule height
10pt}\right\}_{n\ge0}\,.
\end{eqnarray*}
 Following \cite{CW76} (see page 16), set $\beta_{0,1}:=1$ and
for $\bar m>1$ and for $0\le j\le\bar m-1$, define constants
$\beta_{j,\bar m}$ to satisfy the identity:
$$
\prod_{k=1}^{2\bar m-2}\left(s+k-\bar m+\frac12\right)=
\prod_{j=\frac12}^{\bar m-\frac32}(s^2-j^2)=\sum_{j=0}^{\bar m-1}\beta_{j,\bar m}s^{2j}
$$
where the product runs through the half integers that are not integers. We have:
$$
(-1)^{j+\bar m-1}\beta_{j,\bar m}>0\,.
$$
Following \cite{CW76} (see page 17), we may express
$a_n(\Delta_{S^{2\bar m},g})$ for any $n\ge \bar m$ in the form:
\begin{eqnarray*}
a_n&=&\frac{(4\pi)^{\bar m}4^{\bar m-n}}{(2\bar m-1)!}
    \sum_{k=0}^{\bar m-1}\frac{(\bar m-{\textstyle
\frac12})^{n-\bar m+2k+2}k!}{(n-\bar m+k+1)!}\beta_{k,\bar
m}\\ 
&+&\frac{(4\pi)^{\bar m}4^{\bar m-n}}{(2\bar
m-1)!}\sum_{k=0}^{n-\bar m}
\sum_{j=0}^{\bar m-1}(-1)^jc_{j+k}\beta_{j,\bar m}
\frac{(\bar m-\frac12)^{n-\bar m-2k}}{k!(n-\bar m-k)!}\,.
\end{eqnarray*}
There are $\bar m-1$ terms in the first sum and they tend to
zero as
$n\rightarrow\infty$. They play no role in establishing either the lower bound
or the upper bound. The terms in the double sum $\{0\le k\le n-\bar m,0\le
j\le\bar m-1\}$ all have the same sign and thus do not cancel; this is a crucial
point. They are positive if $\bar m$ is odd and negative if $\bar
m$ is even. We may bound
$k!(n-\bar m-k)!\le n!$. On the other hand, if we take
$k=n-\bar m$ and
$j=\bar m-1$, then
\begin{equation}\label{eqn-3.a}
\displaystyle\frac{c_{n-1}}{(n-\bar m)!(\bar m-1)!}\ge
\frac{c_{n-1}}{n!}\ge \pi^{-2n}n!\,.
\end{equation}
Consequently the growth of the terms in equation~(\ref{eqn-3.a}) is at least
$\pi^{-2n}n!$ and at most $\kappa_m n^2\pi^{-2n}n!$ where we bound the
$\beta_{j,\bar m}$ by $\kappa_m$ for some universal constant
$\kappa_n$. The desired estimate now follows.\end{proof}

\subsection{Complex projective space}\label{sect-3.3}

\begin{lemma}
Let $\mathcal{M}$ be complex projective space
$\mathbb{CP}^{\bar m}$ where $\bar m\ge2$. Then
$$
\left\{(-1)^{\bar m-1}a_n(\mathcal{M})\vphantom{\vrule height 10pt}\right\}_{n\ge0}
\approx\left\{\pi^{-2n}(\bar m+1)^{n}\cdot n!\vphantom{\vrule height
10pt}\right\}_{n\ge0}\,.
$$
\end{lemma}

\begin{proof} We follow the discussion in \cite{CW76} pages 18-19. We
distinguish two cases:
\medbreak\noindent{\bf Case I.} Let $\bar m$ be odd. We define
constants $\gamma_{\ell,\bar m}$ using the relation:
$$\prod_{k=1}^{\bar m-1}\left(s+k-\frac{\bar m}2\right)^2=
\prod_{j=\frac12}^{\frac{\bar m}2-1}(s^2-j^2)^2
=\sum_{\ell=0}^{\bar m-1}\gamma_{\ell,\bar m}s^{2\ell}\,.$$
Again there is a parity constraint on these variables. Suppose
we set $\tilde s:=\sqrt{-1}s$. We would then have
$$\prod_{j=\frac12}^{\frac{\bar m}2-1}(\tilde s^2+j^2)^2
=\sum_{\ell=0}^{\bar m-1}\gamma_{\ell,\bar m}(-1)^\ell\tilde s^{2\ell}\,.$$
It is clear that the coefficients of $\tilde s^\ell$ must all be
positive; consequently
$$(-1)^\ell\gamma_{\ell,\bar m}\ge1\quad\text{for}\quad 
0\le \ell\le\bar m-1\,.$$
If $n\ge\bar m-1$, we have:
\begin{eqnarray*}
a_n&=&\frac{(4\pi)^{\bar m-1}(\bar m+1)^{\bar m-n-1}}
   {\bar m!(\bar m-1)!}
   \sum_{j=0}^{\bar m-1}\frac{j!\cdot\gamma_{j,\bar m}}
   {(n-\bar m+2+j)!}\left(\frac{\bar m}2\right)^{2(n-\bar m+2+j)}\\
&+&\frac{(4\pi)^{\bar m-1}(\bar m+1)^{n-\bar m+1}}
   {\bar m!(\bar m-1)!}
   \sum_{k=0}^{n-\bar m+1}\frac{\bar m^{2k}}
   {4^k(\bar m+1)^{2k}}\\
&&\qquad\times\sum_{j=0}^{\bar m-1}(-1)^j
\frac{\gamma_{j,\bar m}c_{n-\bar m+1-k+j}}{k!(n-\bar
m+1-k)!}\,.
\end{eqnarray*}
The terms in the sum $0\le k\le m-2$ tend to zero as
$n\rightarrow\infty$ and play no role. The terms in the sum $\{0\le k\le n-\bar
m+1,0\le j\le\bar m-1\}$ are all positive and do not cancel. The dominant
term arises when
$k=0$ and
$j=\bar m-2$. The desired estimate now follows exactly as in the case of the
even dimensional spheres.

\smallbreak\noindent{\bf Case II.} Let $\bar m$ be even. We again
consider the generating function:
$$\prod_{k=1}^{\bar m-1}\left(s+k-\frac{\bar m}2\right)=\prod_{j=0}^{\frac{\bar
m}2-1}\left(s^2-j^2\right)^2 =\sum_{k=0}^{\bar m-1}\gamma_{k,\bar m}s^{2k}\,.$$
The same argument as that used in Case I shows that
$(-1)^{k-1}\gamma_{k,\bar m}\ge1$ for all $k$. Following \cite{CW76}
(see page 14), we set 
$$d_n=(-1)^nB_{2n+2}/(n+1)\,.$$ 
If $n\ge\bar m-1$, we have (see
\cite{CW76} page 19) that:
\begin{eqnarray*}
a_n&=&\frac{(4\pi)^{\bar m-1}(\bar m+1)^{\bar m-n-1}}
   {\bar m!(\bar m-1)!}\sum_{k=0}^{\bar m-1}\frac{k!\cdot
{\bar m}^{2(n-\bar m+2+k)}\gamma_{k,\bar m}}
    {(n-\bar m+2+k)!4^{n-\bar m+2+k}}\\
&+&\frac{(4\pi)^{\bar m-1}(\bar m+1)^{n-\bar m+1}}
   {\bar m!(\bar m-1)!}\sum_{k=0}^{\bar m-1}\left(\frac{\bar
m^2}{4(\bar m+1)^2}\right)^k\nonumber\\ 
&&\qquad\times
\sum_{j=0}^{\bar m-1}(-1)^j(\bar m+1)^{-k}
\frac{\gamma_{j,\bar m}d_{n-\bar m+1-k+j}}
{k!(n-\bar m+1-k)!}\,.
\end{eqnarray*}
As before, the terms in the first summation contribute
nothing to the analysis. The terms in the second summation
are all negative; the dominant term is obtained by taking $k=0$ and $j=\bar m-1$
to derive the desired lower bound.
\end{proof}

\subsection{Quaternionic projective space}\label{sect-3.4}

\begin{lemma}\label{lem-3.3} Let $\mathcal{M}=\mathbb{HP}^{\bar
m}$be
quaternionic projective space where
$\bar m\ge2$. Then
$$\left\{-a_n(\mathcal{M})\vphantom{\vrule height
10pt}\right\}_{n\ge0}\approx\left\{\pi^{-2n}\cdot n!\vphantom{\vrule height
10pt}\right\}_{n\ge0}\,.$$
\end{lemma}

\begin{proof} We now consider the generating function
$$\prod_{j=\frac12}^{\bar m-\frac32}(s^2-j^2)
\cdot\prod_{j=\frac12}^{\bar m-\frac52}(s^2-j^2)
=\sum_{k=0}^{2\bar m-3}\delta_{k,\bar m}s^{2k}\,.$$
If we set $\tilde s=\sqrt{-1}s$, we can rewrite this in
the form:
\begin{eqnarray*}
&&\prod_{j=\frac12}^{\bar m-\frac32}(-\tilde s^2-j^2)
\cdot\prod_{j=\frac12}^{\bar m-\frac52}(-\tilde s^2-j^2)
=-\prod_{j=\frac12}^{\bar m-\frac32}(\tilde s^2+j^2)
\cdot\prod_{j=\frac12}^{\bar m-\frac52}(\tilde s^2+j^2)\\
&=&\sum_{k=0}^{2\bar m-3}\delta_{k,\bar m}\tilde s^{2k}\quad
\text{so}\quad (-1)^{k+1}\delta_{k,\bar m}\ge1\,.
\end{eqnarray*}
We then have (see \cite{CW76} page 20)
 for $n\ge 2\bar m-2$ that:
\begin{eqnarray*}
a_n(\mathcal{M})&=&\frac{(4\pi)^{2\bar m-2}}{(2\bar m-1)!(2\bar m-3)!}
\sum_{k=0}^{2\bar m-3}\left(\frac{(\bar m-\frac12)^2}{2(\bar
m+1)}\right)^{2(n+2\bar m-3-k)}\\
&&\qquad\times
\frac{k!}{(n+2\bar m-3-k)!}\delta_{k,\bar m}\\
&+&\frac{(4\pi)^{2\bar m-2}}{(2\bar m-1)!(2\bar m-3)!}
\sum_{k=0}^{n-2\bar m+2}
    \frac{(\bar m-\frac12)^{2k}}{2^k(\bar m+1)^k}\\
&&\quad\times\sum_{j=0}^{2\bar m-3}(-1)^j\delta_{j,\bar m}
\frac{c_{j+n-k}}{k!(n-k)!}\,.
\end{eqnarray*}
The terms in the first sum play no role; the terms in the double summation are
all negative and thus do not cancel. We take $k=0$ and $j=2\bar m-3$ to obtain
the desired estimate as before; this is the dominant term.\end{proof}

\subsection{Cayley Plane}\label{sect-3.5}
\begin{lemma}\label{lem-3.5}
Let $\mathcal{M}$ be the Cayley plane $\mathbb{OP}^2$. Then
$$\left\{-a_n(\Delta_{\mathcal{M}})\vphantom{\vrule height
11pt}\right\}_{n\ge0}\approx\left\{\pi^{-2n}\cdot n!\vphantom{\vrule height
11pt}\right\}_{n\ge0}\,.$$
\end{lemma}
\begin{proof}
Following \cite{CW76} (page 20), we define constants $\eta_i$
for $i=0,1,...,7$ by setting:
$$\begin{array}{llll}
\eta_7:=1,&\eta_6:=-\frac{170}4,&\eta_5:=\frac{10437}{16},
   &\eta_4=-\frac{262075}{64},\\
\eta_3:=\frac{2858418}{256},&\eta_2=-\frac{13020525}{1024},
  &\eta_1:=\frac{18455239}{4096},
&\eta_0=-\frac{8037225}{16384}\,.
\vphantom{\vrule height 11pt}\end{array}$$
 The crucial point is that
$(-1)^{i+1}\eta_i\ge1$. For $n\ge7$, one has \cite{CW76} that:
\begin{eqnarray*}
a_n&=&\frac{3!}{7!11!}(4\pi)^8
\sum_{k=0}^7\left(\frac{121}{72}\right)^{n+7-k}\frac{\eta_kk!}
{(n+7-k)!}\\
&+&\frac{3!}{7!11!}(4\pi)^8\sum_{k=0}^{n-8}\left(\frac{121}{72}
\right)^k
\sum_{j=0}^7(-1)^j\frac{\eta_jc_{j+n-k}}{k!(n-k)!}\,.
\end{eqnarray*}
We argue as before to complete the proof of
Lemma~\ref{lem-3.5} and thereby also complete the proof of Theorem~\ref{thm-1.4}
as well.
\end{proof}

\section{Symmetric spaces  with polynomial Plancherel measure}
\label{sect-4}

Lemma~\ref{lem-1.5} (3) permits us to pass between 
$\mathcal{M}$ and the dual manifold $\tilde{\mathcal{M}}$. Thus it suffices to
consider symmetric spaces of non-compact type to establish
Theorem~\ref{thm-1.6}. The heat trace coefficients were determined in \cite{M80}
for quite general operators of Laplace type acting on the space of
smooth sections  of a locally homogeneous vector bundle over an
arbitrary locally symmetric space 
$\mathcal{M}$ of strictly negative curvature. These results were
extended \cite{Go87}, in the spherical
case, to all irreducible, non compact, symmetric spaces of 
higher rank and classical type. Let  $\mathcal{M}=\Gamma\backslash
G/K$ where $G$ is a non compact semi-simple Lie group, $K$ is a
maximal compact subgroup and $\Gamma$  is a uniform lattice in
$G$, that is, a discrete, co-compact subgroup of $G$.  We also
restrict to the scalar Laplacian although in principle these
methods could treat the bundle Laplacian as well.

We adopt the following notational conventions. Let $\mathfrak g$
and $\mathfrak k$ be the Lie algebras of $G$ and $K$ respectively.
We take a Cartan decomposition $\mathfrak g= \mathfrak k \oplus
\mathfrak p$ of $\mathfrak g$; let $\theta$ be the Cartan
involution. We fix a maximal abelian subalgebra  $\mathfrak
a\subset \mathfrak p$.  The Killing form $B_\mathfrak g(.,.)$ of
$\mathfrak g$ induces an inner product on $\mathfrak g$
given by $\langle X,Y\rangle:= -B_\mathfrak g(X, \theta Y)$; we take the dual
inner product on the dual space,
$\mathfrak a^*$.

The Plancherel theorem and the Selberg trace formula play a
main role in the proof of the results in \cite{M80} and in \cite{Go87}.  Let $\mu_G$ denote
Plancherel measure of $G$. If $G$ has rank $1$, then $\mu_G(\lambda) =
p_G(\lambda )f_G(\lambda)$ where
$p_G$ is a polynomial of degree $m-1$ with $m =
\dim(G/K)$ and
$f_G(\lambda)$ is either 1, or $\tanh
(\eta(\lambda))$ or $\coth (\eta(\lambda))$ where $\eta\in\mathfrak{a}^*$. For
groups of arbitrary rank, $\mu_G$ is  a product of Plancherel
measures associated to rank one subgroups corresponding to each
indivisible restricted root of $\mathfrak g$. We restrict to the case
$f_G=1$, i.e. $\mu_G$ is a polynomial function. We have:

\begin{proposition}\label{prop-4.2}
Let $G/K$ be a simply connected irreducible symmetric space of
non-compact type where $\mu_G$ is a polynomial. Then $G/K$ is one of the
following:
\begin{enumerate}
\item Let $\mathcal{M}=H^{2\bar m+1}
=\operatorname{SO}(2\bar m+1,1)/\operatorname{SO}(2\bar m+1)$ so
$G=\operatorname{SO}(2\bar m+1,1)$. Let $\lambda =\lambda_1
\alpha$, and let $\alpha$ be the simple restricted  root of
$(\mathfrak g, \mathfrak a)$. Then:
$$p_G(\lambda) = C_G\prod_{0\le h \le\bar m} 
\left(\lambda_1 ^2 + h^2\right)\,.$$ 
\smallbreak\item 
Let $\mathcal{M}=\operatorname{SU}^*(2\bar
m)/\operatorname{Sp}(\bar m)$ so $G=\operatorname{SU}^*(2\bar
m)$. Adopt the notation of \cite{Go87}. If $\lambda \in \mathfrak a^*$, then:
$$p_G(\lambda) = C_G
\prod _{1\le i < j \le\bar m+1} (\lambda_i -\lambda_j)^2
\,\left((\lambda_i -\lambda_j)^2 +1\right)\,.$$
\smallbreak\item 
Let $\mathcal{M}=E_6^{IV}/F_4$ so $G=E_6^{IV}$.  If $\lambda \in \mathfrak a^*$,
then:
$$p_G(\lambda) = C_G
\prod _{1\le i < j \le 3} \prod_{0\le h\le 3} 
\left((\lambda_i -\lambda_j)^2 + h^2\right)\,.$$
\smallbreak\item 
Let $\mathcal{M}=G/G_u$ where $G$ is a complex simple Lie group
looked on as a real Lie group  and where $G_u$ is a compact real
form of
$G$. Let  $\rho= \frac 12\sum_{\alpha \in \Delta_+} \alpha$
and let $\Delta_+$ be the set of positive roots of  $\mathfrak g$.
 If $\lambda \in \mathfrak a^*$, then:
$$p_G(\lambda) = C_G
\prod _{\alpha\in \Delta_+} \langle \lambda+\rho,\alpha
\rangle^2\,.$$
\end{enumerate}
\end{proposition}

\begin{remark}\rm
Let $d_G$ be the degree of the polynomial
$p_G(\lambda)$;
$$d_G=\dim(G/K)-\textrm{rank}(G/K)\,.$$
In cases (1)-(4) above, $\deg p_G(\lambda)$ equals $2\bar m+2$,
$2\bar m(\bar m+1)$, $24$ and $\#\Delta_+$, respectively. In
particular
$d_G$ is always even.
\end{remark}

Proposition~\ref{prop-4.2} can be established by using the classification of
real simple Lie algebras (see for instance  \cite{He} p. 518 and p. 532). The
explicit form of the Plancherel measure follows by reduction to the rank one
case, by using the Gindikin-Karpelevic formula. However we will not make
use of the precise form of the polynomial in what follows.
Theorem~\ref{thm-1.6} will follow from the following result:

\begin{theorem}
If the Plancherel measure $\mu_G(\lambda)$ is polynomial, then 
there is a polynomial   $\mathcal{P}_M(t) $ of degree $\frac{m-r}2$
with $\mathcal{P}_{\mathcal{M}}(0)=1$ so that
$\mathcal{H}_{\mathcal{M}}(t)=
e^{-t\langle\rho,\rho\rangle}\mathcal{P}_{\mathcal{M}}(t)$.
Consequently,
$$\left\{|a_n(\mathcal{M})|\vphantom{\vrule
height 10pt}\right\}_{n\ge0}\approx
\left\{|\langle\rho,\rho\rangle|^n\frac1{n!}\vphantom{\vrule height
10pt}\right\}_{n\ge0}\,.$$
\end{theorem} 
 
\begin{proof}

The approach (see \cite{Go87} or \cite{M80})   
can be summarized as follows. By using the Selberg trace formula,
$\Tr_{L^2}\{e^{-t\Delta_{\mathcal{M}}}\}$ can be expressed as
a sum of orbital integrals of a function $h_t$ on $G$ defined by
means of spherical inversion. Up to a multiplicative constant,
$$
h_t(x) = \int_{\mathfrak a^*}\phi_\lambda(x) 
e^{-t(\langle
\lambda,\lambda\rangle+\langle\rho,\rho\rangle)}
\mu_G(\lambda)\,d\lambda\,.
$$ 

Now, by the Selberg trace formula, one has that: 
$$  
\Tr_{L^2}\{
e^{-t\Delta}\} = \operatorname{Vol}(\Gamma\backslash G/K)h_t(1)+
\sum_{[\gamma]\in [\Gamma]} \operatorname{Vol}(\Gamma_{\gamma}\backslash
G_{\gamma})\,I_\gamma(h_t)
$$  
where the sum is over the conjugacy classes of $\Gamma$, $G_\gamma$ and
$\Gamma_\gamma$ are the centralizers of
$\gamma$ in $G$ and $\Gamma$ respectively, and $I_\gamma(h_t)$ is
the $\gamma$-orbital integral of $h_t$.   One first shows that the
infinite sum in the right hand-side is asymptotic to $0$ so it
suffices to determine the asymptotic expansion of
$$ 
h_t(1)=
e^{t\langle\rho,\rho\rangle}\int_{\mathfrak a^*}  e^{-t\langle
\lambda,\lambda\rangle}\mu_G(\lambda)\,d\lambda\,.
$$

Fix an orthonormal basis $\{f_1,\dots, f_r\}$ for
${\mathfrak a}^*$. Expand
$$p_G(\lambda) = p_G(\sum_{j=1}^r
\lambda_j f_j) = \sum a_{i_1,\dots,i_r}
\lambda_1^{i_1}\dots \lambda_r^{i_r}\,.$$ If $i_j$ is even we put
$i_j = 2 h_j$ with $h_j \in {\mathbb N}_0$. 
Since  
$\int_{\mathbb R}\lambda^{h} e^{- {\lambda}^2}\,
d\lambda = \Gamma \left(\tfrac h2+\tfrac 12\right),$
we have that:
\medbreak\qquad
$\displaystyle e^{-t{\langle \rho,\rho \rangle}}\int_{\mathfrak a^*}
p_G(\lambda)e^{-t\langle \lambda,\lambda
\rangle} d\lambda$
\smallbreak\qquad\quad
$\displaystyle =  e^{-t{\langle \rho,\rho \rangle}}
\sum_{i_1, \dots, i_r \textrm{ even}} a_{i_1,\dots,i_r} \prod_{1\le j \le r} 
\int_{\mathbb R} \lambda_j^{i_j}
e^{-t{\lambda_j}^2}\, d\lambda_j$
\smallbreak\qquad\quad
$=\displaystyle
t^{-r/2}e^{-t{\langle \rho,\rho \rangle}}\sum_{i_1, \dots, i_r
\textrm{ even}} a_{i_1,\dots,i_r}  \prod_{1\le
j \le r} t^{-\frac {i_j} 2} \int_{\mathbb R} \lambda_j^{i_j} e^{-{\lambda_j}^2}\, d\lambda_j$
\smallbreak\qquad\quad
$=\displaystyle 
{\pi}^{\frac r2}t^{-m/2}e^{-t{\langle \rho,\rho \rangle}}\sum_{h_1, \dots, h_r }    
 a_{2h_1,\dots,2h_r} \prod_{1\le j \le r}\Gamma(h_j+ \tfrac 12) 
 \, t^{\frac {m-r}2-\sum _1^r h_j}$
\smallbreak\qquad\quad
$=\displaystyle
{\pi}^{\frac r2}t^{-m/2}e^{-t{\langle \rho,\rho \rangle}}\sum_{h=0}^{\frac{m-r}2} 
\Big(\sum_{h_1+\ldots +h_r = h}   a_{2h_1,\dots,2h_r} \prod_{1\le j \le r}\Gamma(h_j+ \tfrac 12)\Big)\, t^{\frac{m-r}2 - h}
$
\medbreak\noindent
 By making the change of variables $h'=\frac{m-r}2-h$, we may  complete the proof
by setting:
\medbreak\noindent\hfill
$\displaystyle\mathcal{P}_{\mathcal{M}}(t):= \pi^{\frac r2} \sum_{h=0}^{\frac{m-r}2}
\Big(\sum_{h_1+ \dots h_r = h}  a_{2h_1,\dots,2h_r} \prod_{1\le j \le r}\Gamma(h_j+ \tfrac
12)\Big)\,t^{h}$.\hfill\end{proof}

\section*{Acknowledgments}
The research of  P. Gilkey was partially supported by project MTM2009-07756 
(Spain) and by DFG PI 158/4-6 (Germany).
The research of R. Miatello was partially supported by grants of Conicet, ANPCyT and Secyt-UNC.

\end{document}